\documentclass[a4paper,11pt,twoside,reqno]{amsart}

\usepackage[plainpages=false,pdfpagelabels=true]{hyperref}
\usepackage{amssymb,amsthm}
\usepackage[margin=1in]{geometry}

\newtheorem{Satz}{Theorem}[section]
\newtheorem{Prop}[Satz]{Proposition}
\newtheorem{Lem}[Satz]{Lemma}

\newtheorem{Cor}[Satz]{Corollary}
\newcommand{\grad}{{\operatorname{grad}}}
\newcommand{\vol}{{\operatorname{Vol}}}

\theoremstyle{definition}
\newtheorem{Dfn}[Satz]{Definition}
\newtheorem{Bem}[Satz]{Remark}

\newcommand{\dg}{{\operatorname{div}_g}}
\newcommand{\dv}{\text{ }dV}
\allowdisplaybreaks[1]

\renewcommand{\epsilon}{\varepsilon}

\newcommand{\R}{\ensuremath{\mathbb{R}}}

\numberwithin{equation}{section}

\title{The Ricci flow with metric torsion on closed surfaces}
\author{Volker Branding}
\date{\today}
\address{TU Wien\\
Institut für diskrete Mathematik und Geometrie\\
Wiedner Hauptstraße 8–10, A-1040 Wien}
\email[]{volker@geometrie.tuwien.ac.at}
\begin{document}
\author{Klaus Kr\"oncke}
\address{Universität Hamburg\\
Fachbereich Mathematik
Bundesstraße 55,
20146 Hamburg}
\email[]{klaus.kroencke@uni-hamburg.de}
\begin{abstract}
The famous Uniformization Theorem states that on closed Riemannian surfaces there 
always exists a metric of constant curvature for the Levi-Cevita connection.
In this article we prove that an analogue of the uniformization theorem also holds for connections with metric torsion in the case of non-positive Euler characteristic.
Our main tool is an adapted form of the Ricci flow.
\end{abstract} 

\maketitle

\section{Introduction and results}
Throughout this article we consider a closed Riemannian surface \(M\) with a Riemannian metric \(g\). By \(\nabla^{\scriptscriptstyle LC}\) we denote
the Levi-Cevita connection, which we obtain from the metric \(g\).
For any affine connection there exists a \((2,1)\)-tensor field \(A\) such that
\begin{equation*}
\label{nabla-torsion}
\nabla^{\scriptscriptstyle Tor}_XY=\nabla^{\scriptscriptstyle LC}_XY+A(X,Y)
\end{equation*}
for all vector fields \(X,Y\).
We demand that the connection is \emph{metric}, that is for all vector fields \(X,Y,Z\) one has
\begin{equation*}
\label{nabla-metric}
X(g(Y,Z))=g(\nabla_XY,Z)+g(Y,\nabla_XZ).
\end{equation*}
Hence the endomorphism \(A\) has to skew-adjoint, that is
\begin{equation*}
\label{torsion-skewadjointess}
g(A(X,Y),Z)=-g(Y,A(X,Z)).
\end{equation*}
The skew-adjoint endomorphism \(A\) is called the torsion of the connection.
Since we are only considering a two-dimensional manifold, the torsion endomorphism takes
a rather simple form (see \cite[Theorem 3.1]{MR712664}), that is
\begin{align*}
A(X,Y)=g(X,Y)V-g(V,Y)X,
\end{align*}
where \(V\) is a vector field. For this reason one calls \(A(X,Y)\) \emph{vectorial torsion}.
We obtain the torsion tensor
\begin{align}
\label{torsion-t}T(X,Y)=&A(X,Y)-A(Y,X)=g(V,X)Y-g(V,Y)X.
\end{align}
For more details on metric connections with torsion see \cite{MR3493217} and references therein.
The torsion contributes to the scalar curvature by
\begin{equation}
R=R_g+2\dg V,
\end{equation}
where \(R_g\) denotes the curvature of the Levi-Cevita connection. For a proof, see \cite[Lemma 3.1]{MR3457391}.
For more details on the geometry of metric connections with vectorial torsion see \cite{MR3457391}.

The Gauss-Bonnet formula is not changed by the presence
of torsion since the torsion contributes to the scalar curvature via a divergence term, that is
\begin{align*}
4\pi\chi(M)=\int_MRd\mu=\int_M(R_g+2\dg V)d\mu.
\end{align*}
Let $(M,g)$ be a surface with torsion $A$ induced by the vector field $V$. 
Then we see that for the conformal metric $(M,e^{u}g)$, the vector field $e^{-u}V$ induces a connection with the same torsion $A$.

We are now interested in the following problem: Given a vectorial torsion \(A\) can we find a metric \(g\) in a given conformal class such that
\(R=\mathrm{const}\)? 
Our problem is related to the so-called \emph{prescribed curvature problem}:
Let \((M,h)\) be a closed Riemannian surface and let \(f\colon M\to\R\) be a prescribed function.
Does there exist a metric conformal to \(h\), that is \(g=e^{u}h\), whose scalar curvature is given by \(f\)?
Using the formula for the change of the scalar curvature under a conformal change of the metric this leads to 
the partial differential equation
\begin{equation*}
-\Delta_hu+R_h=fe^{u}.
\end{equation*}
On \(S^2\) this problem is very subtle and became known as \emph{Nierenbergs problem},
it can be attacked using variational techniques \cite{MR908146}.
However, the prescribed curvature problem can also be investigated by a flow approach, see
\cite{MR2137948} and \cite{MR2064965}.

In our case we want to prescribe the function \(r-2\dg V\) (where $r$ is defined below)
which yields the following partial differential equation
\begin{equation*}
-\Delta_hu+R_h=(r-2\dg V)e^{u}.
\end{equation*}
Note that the divergence of the vector field \(V\) depends on the metric \(g\),
hence the function we want to prescribe also depends on the metric.

To study the existence of constant curvature metrics on closed surfaces
for a metric connection with torsion we will make use of the following
evolution equation for the metric:

\begin{align}
\label{ricci-flow}
\frac{\partial}{\partial t}g(t)=(r-R_g-2\dg V_{g(t)})g(t),\qquad g(0)=g_0
\end{align}
with 
\begin{equation*}
r:=\frac{\int_MRd\mu}{\int_M d\mu}=\frac{4\pi\chi(M)}{\vol(M,g)}.
\end{equation*}
In order to ensure that the torsion does not change, the vector field $V_{g(t)}$ has to evolve by $\frac{d}{dt}V=(r-R)V$ or equivalently, the corresponding family of one-forms dual to $V$ via the metric $g(t)$ is constant in time.

A solution of this equation is called a \emph{normalized Ricci flow for a metric connection with (fixed) torsion}. Note that $r$ and the volume are constant along this flow.
In the case of vanishing torsion, it is known that the normalized Ricci flow on any surface converges to a metric of constant curvature \cite{MR954419}.
We will prove the following
\begin{Satz}\label{main-result}
Let \((M,g_0)\) be a closed Riemannian surface.
Then there exists a unique solution \(g(t)\) of \eqref{ricci-flow}
for all \(t\in [0,\infty)\).
If \(\chi(M)\leq 0\) the solution converges to a metric of constant curvature with torsion
as \(t\to\infty\).
\end{Satz}
\begin{Cor}
Let \((M,g)\) be a closed Riemannian surface with  \(\chi(M)\leq 0\) equipped with a vectorial torsion $A$. Then in the conformal class of $g$, there exists a (up to rescaling) unique metric $\bar{g}$ whose scalar curvature with respect to the torsion $A$ is constant.
\end{Cor}
As for the standard Ricci flow the case of positive Euler characteristic is the most difficult one.

We want to point out that the evolution equation \eqref{ricci-flow} is also connected to the so-called 
renormalization group flow studied in quantum field theory. This flow arises in the perturbative
quantization of the so-called \emph{non-linear sigma model}, the Ricci flow being the first order truncation.
Our flow would describe the first order approximation of the renormalization group flow for a 
non-linear sigma model, where the target is two-dimensional and has a metric connection with torsion.
A similar flow in higher dimensions has been studied in \cite{MR2426247}.

\section{Ricci flow with metric torsion}
In this section we discuss the basic properties of \eqref{ricci-flow}.
Moreover, we establish the longtime-existence of our flow.
Rewriting \eqref{ricci-flow} as an evolution equation for the conformal factor, that is \(g(t)=e^{u(t)}g_0\), yields
\begin{equation}
\label{evolution-conformal-factor}
\frac{\partial u}{\partial t}=\Delta_{g(t)}u-e^{-u(t)}(R_{g_0}+2\mathrm{div}_{g_0}V)+r. 
\end{equation}
Using standard parabolic theory we obtain a unique smooth solution of \eqref{evolution-conformal-factor} 
for \(t\in [0,T)\) satisfying \(u_0=0\).

\begin{Dfn}
We call a metric \(g\) \emph{self-similar solution} of \eqref{ricci-flow} if
there exists a vector field \(X\) such that
\begin{align}
\label{self-similar-solution}
(-R_{g}+\int_M R_{g}dV_{g}-2\mathrm{div}_{g}V)g=L_Xg,\qquad L_XV^{\flat}=0
\end{align}
holds. Here, $V^{\flat}$ is the one-form dual to $V$.
\end{Dfn}
\noindent By the standard variational formulas we obtain
\begin{Lem}\label{variational-formulas}
Let \(g(t)\) be a solution of \eqref{ricci-flow}. Then the following formulas hold
\begin{align*}
\frac{d}{dt}\Delta=(R-r)\Delta, \qquad
\frac{d}{dt}\mathrm{div}=(R-r)\mathrm{div}, \qquad
\frac{d}{dt}d\mu=(R-r)d\mu,
\end{align*}
where the divergence here is the divergence applied to one-forms and \(d\mu\) denotes the volume element.
\end{Lem}
\noindent For the Levi-Cevita connection we have the so-called \emph{Ricci-identity}, which states
\begin{equation}
\label{ricci-identity-lc}
\nabla^{\scriptscriptstyle LC}\Delta f=\Delta\nabla^{\scriptscriptstyle LC} f-\frac{1}{2}R_g\nabla^{\scriptscriptstyle LC} f.
\end{equation}
For a metric connection with torsion we have a more complicated identity:
\begin{Lem}
The Ricci identity for a connection with torsion reads
\begin{align*}
\nabla^{\scriptscriptstyle Tor}_X\Delta f=\Delta^{\scriptscriptstyle Tor}\nabla_X f-\frac{1}{2}R\nabla_X f
+\nabla^{\scriptscriptstyle Tor,2}_{X,V} f+\nabla^{\scriptscriptstyle Tor,2}_{V,X} f-2\Delta f\cdot g(V,X),
\end{align*}
where \(\nabla^{\scriptscriptstyle Tor,2}\) denotes the second covariant derivative involving torsion. Note that the Laplacian on functions does not change in the presence of torsion.
\end{Lem}

\begin{proof}
We do the computation with respect to an orthonormal frame. Then
\begin{align*}
\nabla^{\scriptscriptstyle Tor}_{X}\Delta f=\sum_j \nabla^{\scriptscriptstyle Tor,3}_{X,e_j,e_j}f=\sum_j\nabla^{\scriptscriptstyle Tor,3}_{e_j,X,e_j}f
+\sum_j R^{\scriptscriptstyle Tor}_{X,e_j}\nabla_{e_j} f-\sum \nabla^{\scriptscriptstyle Tor,2}_{T(X,e_j),e_j}f.
\end{align*}
Note that the curvature endomorphism is given by
$R^{\scriptscriptstyle Tor}_{X,Y}=[\nabla^{\scriptscriptstyle Tor}_X,\nabla^{\scriptscriptstyle Tor}_Y]-\nabla^{\scriptscriptstyle Tor}_{[X,Y]}$ 
such that $\nabla^{\scriptscriptstyle Tor,2}_{X,Y}-\nabla^{\scriptscriptstyle Tor,2}_{Y,X}=R^{\scriptscriptstyle Tor}_{X,Y}-\nabla^{\scriptscriptstyle Tor}_{T(X,Y)}$. 
Therefore, we get the torsion term above.

To proceed, we remark that the Hessian of a function is not symmetric since it satisfies
\begin{align*}
\nabla^{\scriptscriptstyle Tor,2}_{X,Y}f-\nabla^{\scriptscriptstyle Tor,2}_{Y,X}f=-\nabla^{\scriptscriptstyle Tor}_{T(X,Y)}f=\nabla^{\scriptscriptstyle Tor}_{g(V,Y)X-g(V,X)Y}f
\end{align*}
and thus,
\begin{align*}
\sum_j\nabla^{\scriptscriptstyle Tor,3}_{e_j,X,e_j}f=\Delta^{\scriptscriptstyle Tor}\nabla^{\scriptscriptstyle Tor}_{X}f+\nabla^{\scriptscriptstyle Tor,2}_{V,X}f-g(V,X)\Delta f.
\end{align*}
To consider the other term, we first remark that the Riemann tensor of the connection is antisymmetric in the last two entries
\begin{align*}
\langle R^{\scriptscriptstyle Tor}_{X,Y}Z,W\rangle=-\langle R^{\scriptscriptstyle Tor}_{X,Y}W,Z\rangle,
\end{align*}
which follows from \cite[Lemma 3.1]{MR3457391}. Because the connection is metric, we have $(\nabla^{\scriptscriptstyle Tor}_X\omega)^{\sharp}=\nabla^{\scriptscriptstyle Tor}_X(\omega^{\sharp})$ 
for any one-form $\omega$ which implies $(R^{\scriptscriptstyle Tor}_{X,Y}\omega)^{\sharp}=R^{\scriptscriptstyle Tor}_{X,Y}(\omega^{\sharp})$ and thus
\begin{align*}
(R^{\scriptscriptstyle Tor}_{X,Y}\omega)(Z)=-\omega(R^{\scriptscriptstyle Tor}_{X,Y}Z).
\end{align*}
Therefore, we have
\begin{align*}
\sum_j R^{\scriptscriptstyle Tor}_{X,e_j}\nabla_{e_j} f=-\nabla_{Ric^{\scriptscriptstyle Tor}(X)}f=-\frac{1}{2}R\nabla_{X}f,
\end{align*}
which completes the proof.
\end{proof}
\begin{Bem}Because of this complicated formula, it turned out that it is more convenient to use the Levi-Cevita connection for higher derivative estimates. 
For the rest of the paper, all covariant derivatives are therefore taken with respect to the Levi-Cevita connection.
\end{Bem}
\begin{Lem}\label{lemma-evolution-R}
Let \((M,g(t))\) be a solution of the normalized Ricci flow with torsion on a closed surface.
Then the scalar curvature $R$ of the connection evolves by
\begin{align}
\label{evolution-R}
\frac{d}{dt}R=\Delta R+R(R-r).
\end{align}
\end{Lem}
\begin{proof}
By the first variation of the Levi-Cevita scalar curvature $R_g$ we have,
\begin{align*}
\frac{d}{dt}R_g=-\Delta h-R_g h,
\end{align*}
where we assume $\frac{d}{dt}{g}=h\cdot g$. Moreover, because the family of one-forms corresponding to $V_{g(t)}$ via $g(t)$ is constant and therefore,
\begin{align*}
\frac{d}{dt}(\mathrm{div}V)=(r-R)\mathrm{div}V
\end{align*}
by Lemma \ref{variational-formulas}.
Inserting $h=-R+r$ and using $R=R_g+2\mathrm{div}V^{\flat}$ yields the formula.
\end{proof}

By the standard Ricci identity \eqref{ricci-identity-lc} and the evolution for the scalar curvature we get the following
\begin{Lem}
Let \((M,g(t))\) be a solution of the normalized Ricci flow with torsion on a closed surface.
Then \(|\nabla R|^2\) evolves as
\begin{align}
\label{evolution-nabla-R}
\frac{d}{dt}|\nabla R|^2=\Delta |\nabla R|^2-2|\nabla^2 R|^2+(4R-3r+2\mathrm{div}V)|\nabla R|^2.
\end{align}
\end{Lem}

\begin{Dfn}
We define the curvature potential \(f\) by
\begin{align}
\label{curvature-potential}
\Delta f=R-r.
\end{align}
\end{Dfn}
The right hand side of \eqref{curvature-potential} has vanishing integral, so there always exists a solution to this equation.

\begin{Lem}
Let \(f_0(x,t)\) be a potential of the curvature for a solution \((M,g(t))\)
of the normalized Ricci flow with torsion on a closed Riemannian surface.
Then there exists a function \(c(t)\) depending only on \(t\) such that
the potential function \(f=f_0+c\) satisfies
\begin{align}
\label{evolution-f}
\frac{\partial f}{\partial t}=\Delta f+rf.
\end{align}
\end{Lem}

\begin{proof}
By differentiating \eqref{curvature-potential} with respect to \(t\) and Lemma \ref{lemma-evolution-R} we obtain
\begin{align*}
\frac{\partial}{\partial t}\Delta f&=(R-r)^2+\Delta\frac{\partial f}{\partial t}=\frac{\partial R}{\partial t}
=\Delta\Delta f+R(R-r),
\end{align*}
which gives
\[
\Delta(\frac{\partial f}{\partial t}-\Delta f-rf)=0.
\]
Since the only harmonic functions on a closed surface are constants, there is a function
\(\gamma(t)\) such that
\[
\frac{\partial}{\partial t}f_0=\Delta f_0+rf_0+\gamma.
\]
The statement follows by setting \(c(t):=-e^{rt}\int_0^te^{r\tau}\gamma(\tau)d\tau\).
\end{proof}

\begin{Lem}
\label{evolution-curvature-potenials}
We have the following evolution equations for the potential function \(f\)
\begin{align}
\nonumber\frac{d}{dt}(f^2)&=\Delta(f^2)-2|\nabla f|^2+2rf^2,\\
\label{evolution-nablaf}\frac{d}{dt}|\nabla f|^2&=\Delta |\nabla f|^2-2|\nabla^2 f|^2+2\mathrm{div}V|\nabla f|^2+r|\nabla f|^2.
\end{align}
\end{Lem}
\begin{proof}
The first equation follows by a direct calculation, for the second one we make use of the 
standard Ricci identity \eqref{ricci-identity-lc}, that is
\begin{align*}
\frac{d}{dt}|\nabla f|^2&=2\langle \nabla \Delta f,\nabla f\rangle+2r|\nabla f|^2+(R-r)|\nabla f|^2\\
&=2\langle\Delta \nabla f,\nabla f\rangle-R_g|\nabla f|^2+(R+r)|\nabla f|^2\\
&=\Delta |\nabla f|^2-2|\nabla^2 f|^2+2\mathrm{div}V|\nabla f|^2+r|\nabla f|^2.
\end{align*}
\end{proof}

\begin{Cor}
Let \((M,g(t))\) be a solution of the normalized Ricci flow with torsion on a closed surface
and let \(f\) be a potential for the curvature. Then there exists a constant \(C\) such that
\begin{equation}
\label{bound-f}
|f|\leq Ce^{rt}.
\end{equation}
\end{Cor}
\begin{proof}
This follows directly from the maximum principle.
\end{proof}
\noindent Using the evolution equation for the curvature potential \eqref{evolution-f} we find
\begin{Lem}
Let \((M,g(t))\) be a solution of the normalized Ricci flow with torsion on a closed surface.
For \(r\leq 0\) there exists a uniform constant C such that
\[
g\leq C,\qquad |\mathrm{div}V|\leq C.
\]
\end{Lem}
\begin{proof}
We calculate
\begin{align*}
\frac{d}{dt}(\mathrm{div}V)^2=2\Delta f(\mathrm{div}V)^2=2(\frac{df}{dt}-rf)(\mathrm{div V})^2,
\end{align*}  
which can be integrated as
\begin{align}
\label{bound-div-V}
|\mathrm{div}V|=\exp(f(x,t)-f(x,0)-r\int_0^tf(x,\tau)d\tau).
\end{align}
Note that this bound is uniform if $r\leq 0$.
We can rewrite the evolution equation for the metric as
\[
\frac{d}{dt}g=-(\Delta f)g=(\frac{df}{dt}-rf)g
\]
and again the result follows by integration.
\end{proof}

\begin{Prop}
Let \((M,g(t))\) be a solution of the normalized Ricci flow with torsion on a closed surface.
We obtain
\begin{equation}
\label{bound-curvature-t-dependent}
R\leq C,
\end{equation}
where the constant \(C\) depends on \(t\) and is finite for finite values of \(t\).
\end{Prop}
\begin{proof}
We set
\[
Z:=|\nabla f|^2+\beta(R-r),
\]
where \(\beta\) is some positive number. By a direct calculation we obtain
\begin{align*}
\frac{dZ}{\partial t}&=\Delta Z+rZ+2\mathrm{div}V|\nabla f|^2+\beta(\Delta f)^2-2|\nabla^2f|^2 \\ 
&\leq\Delta Z+rZ+2\mathrm{div}V|\nabla f|^2+2|\nabla^2f|^2(\beta-1).
\end{align*}
In the following we choose \(\beta\leq 1\). To bound \(\mathrm{div}V|\nabla f|^2\) we note that
we have a time dependent bound of \(\mathrm{div}V\) from \eqref{bound-div-V} for arbitrary \(r\), denote this bound by \(v_1(t)\).
Note that \(v_1(t)\) is finite for finite values of \(t\). Thus, we obtain from \eqref{evolution-nablaf} that
\[
\frac{d}{dt}|\nabla f|^2\leq\Delta|\nabla f|^2+(r+v_1(t))|\nabla f|^2.
\]
By the maximum principle we obtain
\[
|\nabla f|^2\leq v_2(t),
\]
where \(v_2(t)\) is again finite for finite values of \(t\). Thus, we find
\[
\frac{dZ}{\partial t}\leq\Delta Z+rZ+v_1(t)v_2(t).
\]
Making use of the maximum principle and the Gronwall-Lemma we obtain
\[
Z(t)\leq Z(0)e^{rt}+\int_0^tv_1(s)v_2(s)e^{rs}.
\]
Since the right hand side of this inequality is finite for finite values of \(t\) we obtain the statement.
\end{proof}

\begin{Lem}
Let \((M,g(t))\) be a solution of the normalized Ricci flow with torsion on a closed surface.
\begin{itemize}
 \item If \(r<0\), then
  \[
   R-r\geq (R_{min}(0)-r)e^{rt}.
  \]
 \item If \(r=0\), then
  \[
   R>-\frac{1}{t}.
  \]
 \item If \(r>0\), then
  \[
   R\geq R_{min}(0)e^{-rt}.
  \]
\end{itemize}
In each case the right hand side tends to \(0\) as \(t\to\infty\).
\end{Lem}
\begin{proof}
This follows from \eqref{evolution-R} and the maximum principle exactly as in \cite[Lemma 5.9]{MR2061425}.
\end{proof}

\begin{Prop}
Let \((M,g(t))\) be a solution of the normalized Ricci flow with torsion on a closed surface.
Then there exists a unique solution for all \(t\in[0,\infty)\).
\end{Prop}
\begin{proof}
By \eqref{ricci-flow} we are deforming the metric \(g(t)\) in its conformal class. Writing \(g(t)=e^{u(t)}g_0\)
the equation \eqref{ricci-flow} is equivalent to 
\[
\frac{\partial u}{\partial t}=-R(u)+r.
\]
By the bound \eqref{bound-curvature-t-dependent} this yields a bound on \(\big|\frac{\partial u}{\partial t}|\) and
also \(u(t)\). Moreover, we can rewrite \eqref{ricci-flow} as
\begin{equation}
\label{evolution-u}
\frac{\partial u}{\partial t}=\Delta_{g(t)}u-e^{-u(t)}(R_{g_0}+2\mathrm{div}_{g_0}V) +r
\end{equation}
with the initial condition \(u(0)=0\).
Interpreting \eqref{evolution-u} as an elliptic equation we may apply elliptic Schauder theory (\cite{MR925006}, p. 79)
and get that \(u\in C^{1+\alpha}(M)\). Making use of the regularity gained from elliptic Schauder theory we interpret \eqref{evolution-u}
as a parabolic equation. By application of parabolic Schauder theory (\cite{MR925006}, p. 79) we then get that \(u\in C^{2+\alpha}(M)\).
Suppose that there would be a maximal time of existence \(T_{max}\), then using that \(u\in C^{2+\alpha}(M)\), we can continue
the solution beyond \(T_{max}\) yielding the claim.
\end{proof}

\section{Convergence}
In this section we establish convergence of \eqref{ricci-flow} in the cases \(\chi(M)\leq 0\) and
point out the difficulties of the spherical case.
To this end we will often make use of the following identities
\begin{align*}
[\nabla,\Delta]f=-\frac{1}{2}R_g\nabla f,\qquad [\nabla,\Delta]T=R_g*\nabla T+\nabla R_g* T,
\qquad [\frac{d}{dt},\nabla]T= \nabla R*T
\end{align*}
for functions $f$ and for higher tensors $T$, where the last equation holds along \eqref{ricci-flow}. Here, $*$ is Hamilton's notation for a combination of tensor products and contractions. 
Thus, by induction, it is easily seen that
\begin{align*}
[\nabla^k,\Delta]f=\sum_{l=0}^{k-1}\nabla^lR_g*\nabla^{k-l}f,\qquad [\frac{d}{dt},\nabla^k]f=\sum_{l=1}^{k}\nabla^lR*\nabla^{k-l}f
\end{align*}
and in particular
\begin{align}
\label{nablak-laplace-interchange}
[\nabla^k,\Delta]R=\sum_{l=0}^{k-1}[\nabla^lR*\nabla^{k-l}R+\nabla^l\mathrm{div}V*\nabla^{k-l}R],\qquad \quad [\frac{d}{dt},\nabla^k]R=\sum_{l=1}^{k}\nabla^lR*\nabla^{k-l}R.
\end{align}

\begin{Bem}
If the Euler characteristic is non-positive then constant curvature metrics with torsion are unique up to rescaling.
To see this let $g$ be a metric on a surface and $\bar{g}=e^{u}g$.  
A straightforward calculation shows
\begin{align*}
\bar{R}=R_{\bar{g}}+2\mathrm{div}_{\bar{g}}V_{\bar{g}}=e^{-u}(R_g-\Delta u+2\mathrm{div}_gV_g)=e^{-u}(R-\Delta u).
\end{align*}
If we assume that $R=\bar{R}=r$ (constant scalar curvature and same volume) we get
\begin{align*}
\Delta u=r(e^{-u}-1).
\end{align*}
It is immediate that $u$ must be constant on the torus. If $r$ is negative, the maximum principle and the fact that $\int_M e^{2u}d\mu_g=\vol(M,g)$ also imply that $u$ is constant.
\end{Bem}

\subsection{Negative Euler Characteristic}

In this section we will prove the following subcase of Theorem \ref{main-result}:
\begin{Satz}
\label{theorem-negative-euler}
Let \((M,g_0)\) be a closed Riemannian surface.
If \(\chi(M)<0\) the solution of \eqref{ricci-flow} starting at \(g_0\) converges exponentially in any \(C^k\)-norm
to a metric \(g_\infty\) of constant curvature with torsion as \(t\to\infty\).
\end{Satz}

Our proof follows a similar approach as \cite[Chapter 5]{MR2061425}.
Before we prove the result we make the following observation.
\begin{Bem}
There do not exist conformal Killing vector fields on surfaces with negative Euler characteristic.
For metrics of constant negative curvature this follows from \cite[Theorem 4.44]{MR2371700}.
For arbitrary metrics the claim follows from conformal invariance.
Hence there do not exist a non-trivial self-similar solution of \eqref{ricci-flow} on surfaces with negative Euler characteristic.
\end{Bem}
\noindent We set
\[
G:=|\nabla f|^2+\alpha f^2+\beta (R-r)
\]
with \(\alpha,\beta>0\).
By a direct calculation we obtain the following
\begin{Lem}
Let \((M,g(t))\) be a solution of \eqref{ricci-flow}. Then the quantity \(G\) evolves by 
\begin{equation}
\label{evolution-G}
\frac{d}{dt}G=\Delta G+(2\mathrm{div}V+r-2\alpha)|\nabla f|^2+2\alpha r f^2+\beta r(R-r)+\beta(\Delta f)^2-2|\nabla^2f|^2.
\end{equation}
\end{Lem}
\begin{proof}
This follows from \eqref{evolution-R}, Lemma \ref{evolution-curvature-potenials} and the identity
\begin{align*}
\frac{d}{dt}(R-r)=\Delta(R-r)+r(R-r)+(\Delta f)^2.
\end{align*}
\end{proof}

\begin{Lem}
Let \((M,g(t))\) with \(\chi(M)<0\) be a solution of \eqref{ricci-flow}. Then we have the estimate
\begin{align*}
|R-r|\leq Ce^{r/2t}.
\end{align*}
\end{Lem}
\begin{proof}
We estimate the evolution equation for the quantity \(G\) from \eqref{evolution-G} and find
\begin{align*}
\frac{d}{dt}G\leq \Delta G+(2\mathrm{div}V+r-2\alpha)|\nabla f|^2+2\alpha r f^2 +\beta r(R-r)+2(\beta-1)|\nabla^2f|^2.
\end{align*}
Choosing \(\beta\leq 1\) and \(\alpha\geq\mathrm{div}V-\frac{r}{2}\) we obtain
\begin{align*}
R-r\leq Ce^{rt}
\end{align*}
by the maximum principle.
Moreover, we have the differential inequality
\begin{align*}
\frac{d}{dt}(R-r)^2=\Delta(R-r)^2-2|\nabla R|^2+2R(R-r)^2\leq \Delta(R-r)^2+2r(R-r)^2+Ce^{rt}(R-r)^2.
\end{align*}
In particular, we have
\begin{align*}
\frac{d}{dt}(R-r)^2=\Delta(R-r)^2+r(R-r)^2
\end{align*}
for large \(t\), which yields the result by the maximum principle. 
\end{proof}

\begin{Lem}
Let \((M,g(t))\) with \(\chi(M)<0\) be a solution of \eqref{ricci-flow}. Then we have the estimates
\begin{align*}
|\nabla R|\leq Ce^{r/2t}, \qquad |\nabla\mathrm{div}V|\leq C.
\end{align*}
\end{Lem}
\begin{proof}
Using \eqref{evolution-nabla-R} we get for $\alpha>0$
\begin{align*}
\frac{d}{dt}(|\nabla R|^2+\alpha(R-r)^2)\leq& \Delta(|\nabla R|^2+\alpha(R-r)^2)+(4R-3r+2\mathrm{div}V-2\alpha)|\nabla R|^2 \\
&+2\alpha R(R-r)^2\\
\leq & \Delta(|\nabla R|^2+\alpha(R-r)^2)+r(|\nabla R|^2+\alpha(R-r)^2),
\end{align*}
where we chose \(\alpha\geq 2(R-r)+\mathrm{div}V\) in the second step.
By application of the maximum principle this yields
\begin{align*}
|\nabla R|\leq Ce^{r/2t}.
\end{align*}
Now we deduce that $\nabla\mathrm{div}V$ is bounded. In fact
\begin{align*}
\frac{d}{dt}|\nabla \mathrm{div}V|^2=3(R-r)|\nabla \mathrm{div}V|^2+2\mathrm{div}V\langle \nabla R,\nabla \mathrm{div}V\rangle
\end{align*}
and the bound on $\mathrm{div}V$ and the decay of $R-r$ and $\nabla R$ immediately implies a bound on $\nabla\mathrm{div}V$.
\end{proof}

Making use of the same methods we now establish exponential decay of the \(k\)-th derivative of the curvature.

\begin{Prop}
Let \((M,g(t))\) with \(\chi(M)<0\) be a solution of \eqref{ricci-flow}. Then we have the estimates
\begin{align*}
|\nabla^kR|\leq C_ke^{r/2t}
\end{align*}
for all \(k>0\).
\end{Prop}

\begin{Bem}
Note that this Proposition implies Theorem \ref{theorem-negative-euler}.
\end{Bem}

\begin{proof}
Making use of \eqref{nablak-laplace-interchange} we get
\begin{align*}
\frac{d}{dt}|\nabla^kR|^2= &\Delta |\nabla^kR|^2-2|\nabla^{k+1}R|^2-(k+2)r|\nabla^kR|^2+\sum_{l=0}^k\nabla^lR*\nabla^{k-l}R*\nabla^kR\\&+\sum_{l=1}^{k-1}\nabla^l\mathrm{div}V*\nabla^{k-l}R*\nabla^kR.
\end{align*}
Now we proceed by induction on $k$. Assume that
\begin{align*}
|\nabla^lR|\leq C(l)e^{r/2t}, \qquad |\nabla^l\mathrm{div}V|\leq C(l) \text{ for all } 1\leq l\leq k-1.
\end{align*}
From the above evolution equation and by the induction hypothesis, we have
\begin{align*}
\frac{d}{dt}|\nabla^kR|^2&\leq\Delta |\nabla^kR|^2+(C_1(|R|+|\mathrm{div}V|)+C_2-(k+2)r)|\nabla^kR|^2+e^{rt},\\
\frac{d}{dt}|\nabla^{k-1}R|^2&\leq\Delta|\nabla^{k-1}R|^2-2|\nabla^kR|^2+C_3|\nabla^{k-1}R|^2+C_4e^{r/2t}|\nabla^{k-1}R|\\
&\leq\Delta|\nabla^{k-1}R|^2-2|\nabla^kR|^2+C_4e^{rt}.
\end{align*}
For $\alpha>0$ large enough we find
\begin{align*}
\frac{d}{dt}(|\nabla^kR|^2+\alpha|\nabla^{k-1}R|^2)&\leq \Delta(|\nabla^kR|^2+\alpha|\nabla^{k-1}R|^2)\\&
+(C_1(|R|+|\mathrm{div}V|)+C_2-(k+2)r-2\alpha)|\nabla^kR|^2+C_5e^{rt}\\
&\leq \Delta(|\nabla^kR|^2+\alpha|\nabla^{k-1}R|^2)-(\epsilon-r)(|\nabla^kR|^2+\alpha|\nabla^{k-1}R|^2)+C_6e^{rt}
\end{align*}
and thus by the maximum principle
\begin{align*}
|\nabla^kR|\leq Ce^{r/2t}.
\end{align*}
To finish the induction step, we have to show that $\nabla^k\mathrm{div}V$ is bounded.
From $\frac{d}{dt}\mathrm{div}V=(R-r)\mathrm{div}V$, we get by differentiating
\begin{align*}
\frac{d}{dt}\nabla^k\mathrm{div}V&=[\frac{d}{dt},\nabla^k]\mathrm{div}V+\nabla^k\frac{d}{dt}\mathrm{div}V
=\sum_{l=1}^k\nabla^l(R-r)*\nabla^{k-l}\mathrm{div}V.
\end{align*}
This yields the estimate
\begin{align*}
\frac{d}{dt}|\nabla^k\mathrm{div}V|^2&\leq\sum_{l=1}^kC_l|\nabla^l(R-r)||\nabla^{k-l}\mathrm{div}V||\nabla^k \mathrm{div}V|\leq C_7e^{r/2t}|\nabla^k\mathrm{div}V|,
\end{align*}
which implies the bound of $|\nabla^k\mathrm{div}V|$ and completes the induction step.
\end{proof}

\subsection{Zero Euler Characteristic}
In this section we will prove the remaining subcase of Theorem \ref{main-result}.
\begin{Satz}
\label{theorem-zero-euler}
Let \((M,g_0)\) be a closed Riemannian surface.
If \(\chi(M)=0\) the solution of \eqref{ricci-flow} starting at \(g_0\) converges exponentially
to a metric \(g_\infty\) of constant curvature with torsion as \(t\to\infty\).
\end{Satz}

The proof of Theorem \ref{theorem-zero-euler} is similar to the case of the standard Ricci flow.
By assumption we have zero Euler characteristic and thus the family of metrics \(g(t)\) is bounded uniformly.
First, we derive uniform bounds on the curvature, then we show that several integral norms of the curvature
decay exponentially. By the Sobolev embedding theorem we then conclude that the curvature has to decay exponentially itself.

Before we prove the main result we make the following observations.
\begin{Bem}
On the torus, one can construct plenty of solutions of the equation
\begin{align}
(-R_{g}+r_g-2\mathrm{div}_{g}V)g=L_Xg.
\end{align}
In fact, any non-flat metric on the torus possesses a conformal Killing vector field $X$ that is not Killing, i.e.
$L_Xg=f\cdot g$ for some nonzero $f\in C^{\infty}(M)$ with vanishing integral. Because the integral is vanishing, 
we can always find a vector field $V$ such that the equation $-R_{g}+r_g-2\mathrm{div}_{g}V=f$ is satisfied.

A solution of \eqref{ricci-flow} starting at such a metric $g$ is tangent to the orbit of the diffeomorphism group acting on $g$ at $t=0$
but we do not get a family of isometric metrics since the flow converges to a constant scalar curvature metric as $t\to\infty$ due to the main theorem.
\end{Bem}

\begin{Bem}
If we consider the conformal transformation \(g=e^uh\) for a given metric \(h\) and a function \(u\),
then finding a metric of zero curvature is equivalent to solving the following equation
\[
0=R_g=e^{-u}(-\Delta_hu+R_h+2\mathrm{div}_hV).
\]
However, the Poisson equation
\[
\Delta_hu=R_h+2\mathrm{div}_hV
\]
can always be solved since the right hand side has vanishing integral.
Due to this reason we expect convergence of \eqref{ricci-flow} in the case of zero Euler characteristic.
\end{Bem}

\begin{Lem}
Let \((M,g(t))\) with \(\chi(M)=0\) be a solution of \eqref{ricci-flow}. Then we have the uniform estimate
\begin{equation}
f^2+|\nabla f|^2+R^2+|\nabla R|^2\leq C.
\end{equation}
\end{Lem}
\begin{proof}
We set 
\[
H:=f^2+\alpha|\nabla f|^2+\beta_0R^2+\beta_1|\nabla R|^2
\]
with positive constants \(\alpha,\beta_0,\beta_1\) to be fixed later.
Using the evolution equations for the scalar curvature \eqref{evolution-R}, its derivative \eqref{evolution-nabla-R},
the curvature potential and its derivative \eqref{evolution-nablaf}, we find
\begin{align*}
\frac{d}{dt}H=&\Delta H+|\nabla f|^2(2\alpha\mathrm{div}V-2)+\beta_0R(\Delta f)^2-2\alpha|\nabla^2f|^2 \\
&+|\nabla R|^2(-2\beta_0+\beta_14R+2\beta_1\mathrm{div}V)-2\beta_1|\nabla^2R|^2 \\
\leq &\Delta H+2|\nabla f|^2(\alpha\mathrm{div}V-1)+2|\nabla^2f|^2(\beta_0R-\alpha)
+2|\nabla R|^2(2\beta_1R+\beta_1\mathrm{div}V-\beta_0).
\end{align*}
By adjusting the positive constants \(\alpha,\beta_0\) and \(\beta_1\) and application of the maximum principle 
we obtain the claim.
\end{proof}

Making use of the same method it is not difficult to also bound the \(k\)-th derivative of the curvature.
However, to achieve decay of the curvature we need to consider integral estimates.
By a direct calculation we obtain the following
\begin{Lem}
\label{evolution-integrals-torus}
Let \((M,g(t))\) with \(\chi(M)=0\) be a solution of \eqref{ricci-flow}.
Then the following equations hold:
\begin{align*}
\frac{d}{dt}\int_M|\nabla f|^2d\mu=&-2\int_MR^2d\mu,\\
\frac{d}{dt}\int_MR^2d\mu=&-2\int_M|\nabla R|^2d\mu+\int_M R^3d\mu, \\
\frac{d}{dt}\int_M|\nabla R|^2d\mu=&-2\int_M(\Delta R)^2d\mu-2\int_MR^2\Delta Rd\mu, \\
\frac{d}{dt}\int_M(\Delta R)^2d\mu=&-2\int_M|\nabla\Delta R|^2d\mu-2\int_M \Delta R\Delta R^2d\mu+\int_MR(\Delta R)^2d\mu.
\end{align*}
\end{Lem}
\begin{proof}
The key point is to use the equality $R=\Delta f=\partial_t f$ in a suitable way. To this end we calculate
\begin{align*}
\frac{d}{dt}\int_M|\nabla f|^2d\mu=&\int_MR|\nabla f|^2d\mu+\int_M(-R)|\nabla f|^2d\mu-2\int_M\frac{\partial f}{\partial t}\Delta fd\mu=-2\int_MR^2d\mu.
\end{align*}
By a similarly calculation we find
\begin{align*}
\frac{d}{dt}\int_M R^2d\mu=2\int_MR(\Delta R+R^2)d\mu-\int_M R^3d\mu=-2\int_M|\nabla R|^2d\mu+\int_MR^3d\mu
\end{align*}
and the other two equalities follow similarly.
\end{proof}

By interchanging covariant derivative and making use of the standard Ricci identity \eqref{ricci-identity-lc} we find
\begin{align}
\label{equality-nabla-delta-R}
\int_M|\nabla^2R|^2d\mu=\int_M(\Delta R)^2d\mu-\frac{1}{2}\int_MR_g|\nabla R|^2.
\end{align}
Moreover, we will make use of the following 
\begin{Lem}
Let \(f\) be a smooth function and \(c_1\) some positive constant.
Then the following inequality holds
\begin{align}
\label{inequality-nabla-delta-R}
\int_M(\Delta f)^2d\mu\geq c_1\int_M|\nabla f|^2d\mu.
\end{align}
\end{Lem}
\begin{proof}
We will prove the claim by a spectral decomposition.
Let \(f_i\) be the i-th eigenfunction of the Laplace operator, that is
\(\Delta f_i=-\lambda_if_i\), where \(\lambda_i\) denotes the i-th eigenvalue. 
We expand \(f=\sum_{i=0}^\infty a_if_i\), where \(a_i\) are real-valued coefficients. 
Making use of the \(L^2\)-orthogonality of the eigenfunctions of the Laplace operator 
the claim is equivalent to
\[
\sum_{i=0}^\infty a_i^2(\lambda_i^2-c_1\lambda_i)\geq 0.
\]
Due to the spectral properties of the Laplace operator it is easy to find some positive constant \(c_1\)
such that the above inequality holds, which finishes the proof.
\end{proof}

\begin{Prop}
\label{prop-torus-decay-curvature}
Let \((M,g(t))\) with \(\chi(M)=0\) be a solution of \eqref{ricci-flow}.
Then the scalar curvature converges exponentially to zero as \(t\to\infty\), that is
\[
|R|\leq Ce^{-Ct}.
\]
\end{Prop}
\begin{proof}
First, we fix two positive numbers \(\alpha,\beta\) and calculate
\begin{align*}
\frac{d}{dt}\int_M(\alpha|\nabla f|^2+\beta R^2)d\mu=&-2\alpha\int_MR^2d\mu-\beta\int_M|\nabla R|^2d\mu+\beta\int_MR^3d\mu \\
\leq &-\alpha\int_M R^2d\mu-\alpha\int_M (\Delta f)^2d\mu+\beta C\int_M R^2d\mu\\
\leq &(\beta C-\alpha)\int_MR^2d\mu-c_1\alpha\int_M|\nabla f|^2d\mu,
\end{align*}
where we first used that \(R\) is bounded uniformly and then applied \eqref{inequality-nabla-delta-R} to the curvature potential.
Choosing \(\alpha>\beta C\) we obtain
\[
\int_M(\alpha|\nabla f|^2+\beta R^2)d\mu\leq Ce^{-Ct}.
\]
Using the evolution equation for \(|\nabla R|^2\) from Lemma \ref{evolution-integrals-torus} we find
\begin{align*}
\frac{d}{dt}\int_M|\nabla R|^2d\mu\leq-\int_M(\Delta R)^2d\mu+\int_MR^4d\mu \leq -c_1\int_M|\nabla R|^2d\mu+C\int_MR^2d\mu,
\end{align*}
where we used Young's inequality and \eqref{inequality-nabla-delta-R}. By the Gronwall inequality this implies
\[
\int_M|\nabla R|^2d\mu\leq Ce^{-Ct}.
\]
Using the evolution equation for \((\Delta R)^2\) from Lemma \ref{evolution-integrals-torus} we find
\begin{align*}
\frac{d}{dt}\int_M(\Delta R)^2d\mu=&-2\int_M|\nabla\Delta R|^2d\mu+2\int_M\langle\nabla R^2,\nabla\Delta R\rangle d\mu+\int_MR(\Delta R)^2d\mu\\
&\leq \int_M|\nabla R^2|^2d\mu+C_2\int_M(\Delta R)^2d\mu \\
&\leq C_1\int_M|\nabla R|^2d\mu+C_2\int_M(\Delta R)^2d\mu,
\end{align*}
where \(C_1,C_2\) are positive constants.
Again, we fix two positive numbers \(\alpha,\beta\) and calculate
\begin{align*}
\frac{d}{dt}\int_M(\alpha|\nabla R|^2+\beta(\Delta R)^2)d\mu \leq &-\alpha\int_M(\Delta R)^2d\mu-\alpha c_1\int_M|\nabla R|^2d\mu+\beta C_1\int_M|\nabla R|^2d\mu \\
&+\alpha C\int_M R^2d\mu+\beta C_2\int_M(\Delta R)^2d\mu.
\end{align*}
Now, we choose \(\alpha>\frac{\beta C_1}{c_1}+\beta C_2\), which implies 
\[
\frac{d}{dt}\int_M(\alpha|\nabla R|^2+\beta(\Delta R)^2)d\mu\leq -C\int_M(\alpha|\nabla R|^2+\beta(\Delta R)^2)d\mu+\alpha C\int_M R^2d\mu.
\]
Making use of the Gronwall inequality again and using the exponential decay of the \(L^2\)-norm of \(R^2\), we get
\[
\int_M(\Delta R)^2d\mu\leq Ce^{-Ct}.
\]
Moreover, by \eqref{equality-nabla-delta-R} this implies 
\[
\int_M|\nabla^2R|^2d\mu\leq Ce^{-Ct}.
\]
We now have achieved the exponential convergence of the \(L^2\)-norms of \(R,\nabla R\) and \(\nabla^2R\).
Thus, by the Sobolev embedding theorem, this yields the exponential convergence of \(R\).
\end{proof}

Using the method applied in the proof of Proposition \ref{prop-torus-decay-curvature} one can
also establish the exponential decay of higher derivatives of the curvature.
This finishes the proof of Theorem \ref{theorem-zero-euler}.

\section{Positive Euler Characteristic}
As for the standard Ricci flow the case of the sphere is the most complicated one.
However, several tools have been developed that ensure convergence of the standard Ricci flow.
These tools include the so-called \emph{surface entropy} and a \emph{differential Harnack estimate}.

Unfortunately, it turns out that all of these tools can no longer be applied in our case.
We have already seen that we have longtime existence, but the question of convergence is rather subtle.

\subsection{Linearization Stability and Instability}
\begin{Lem}
Let \(T(g)=(-R+r)g\) and \(\hat g\) such that \(R_{\hat g}=r\), which corresponds to a stationary point of \eqref{ricci-flow}.
Then 
\begin{align*}
\frac{d}{ds}\big|_{s=0}T((1+sh)\hat g)=(\Delta h+rh)\hat g,
\end{align*}
where \(h\in C^\infty(M)\).
\end{Lem}

\begin{proof}
This follows directly from the standard variational formulas
\begin{align*}
\frac{d}{ds}|_{s=0}R_{g+thg}=-\Delta h-R_gh\qquad \frac{d}{ds}|_{s=0}\mathrm{div}_{g+thg}V^{\flat}=-\mathrm{div}V^{\flat}h.
\end{align*}
\end{proof}

\begin{Prop}
On the sphere any stationary point except the round metric is linearly unstable.
\end{Prop}
\begin{proof}
Since the flow preserves the volume, we restrict to volume-preserving deformations of the metric and hence, we assume that $h$ has vanishing integral. 
The linearization operator (restricted to volume-preserving deformations) admits positive eigenvalues whenever $\mathrm{spec}(-\Delta)\cap (0,r)\neq \emptyset$. 
However, due to a theorem by Hersch \cite{hersch1970quatre}, the first nonzero eigenvalue of the Laplacian satisfies
\begin{align*}
\lambda_1\leq \frac{8\pi}{Vol(S^2,g)}=\frac{\int R\dv}{Vol(S^2,g)}=r
\end{align*}
for any metric on $S^2$ and equality only holds for the standard sphere.
\end{proof}

\begin{Bem}
The previous Proposition shows that we cannot expect convergence in general.
Our problem is similar to the prescribed curvature problem on \(S^2\), which can also be attacked by a flow approach \cite{MR2137948}.
In that reference it is shown that not every function can be prescribed on \(S^2\) and although our problem 
is slightly different we encounter similar difficulties here.

On the other hand, $\R P^2$ admits metrics such that $\lambda_1>r$ and thus these metrics are linearly stable. 
It is not clear whether one can expect convergence of the flow \eqref{ricci-flow} on $\R P^2$ in general.
\end{Bem}

\subsection{Self-similar solutions on the sphere}
In this subsection we study self-similar solutions of \eqref{ricci-flow} on the two-sphere,
which is important to understand the asymptotics of our flow.
In the case of the standard Ricci flow the only soliton is the round metric,
which follows by applying the Kazdan-Warner identity. It turns out that this identity 
does not give any information in our case and so we follow a different approach here
using ideas from dynamical systems theory.

\begin{Lem}\label{vconstant}
Let \(g\) be a self-similar solution of \eqref{ricci-flow} and \(v\in C^\infty(M)\) satisfying \(\Delta v=\mathrm{div}V^\flat\).
Then \(v\) is constant along the integral curves of \(X\).
\end{Lem}
\begin{proof}
The one-form $V^{\flat}$ splits as $V^{\flat}=dv+\omega$ where $v\in C^{\infty}(M), \mathrm{div}_g\omega=0$ and $v$ is non-constant since $\mathrm{div}_{g}V$ must be nonzero. 
If $\mathrm{div}_gV$ was zero, we had a self-similar solution of the standard Ricci flow and so it must be constant. Now we have
\begin{align*}0=L_XV^{\flat}=L_Xdv+L_X\omega=d(X(v))+L_X\omega
\end{align*}
and
\begin{align*}
\mathrm{div}_g(L_X\omega)=X(\mathrm{div}_g\omega)-\mathrm{div}_{L_Xg}\omega=-\mathrm{div}_{L_Xg}\omega=-\mathrm{div}_{h\cdot g}\omega=h\cdot\mathrm{div}_g\omega=0.
\end{align*}
Thus, $d(X(v))$ and $L_X\omega$ are both linearly independent hence they both vanish. 
Thus, $X(v)$ is constant and hence it must be equal to zero since this expression vanishes at the extrema of $v$. 
\end{proof}

\begin{Prop}
The round sphere cannot be a self-similar solution of \eqref{ricci-flow}. 
\end{Prop}
\begin{proof}
In this case, $X=\grad h$ and the function $h$ is a restriction of a linear function in $\R^3$. 
Without loss of generality, assume that $h$ is the projection to the $z$-coordinate. 
Then the critical points of $X$ are the poles and the integral curves are the half great-circles connecting the poles.
If $v$ is constant along these curves it must be constant on the whole manifold by continuity which causes a contradiction.
\end{proof}

\begin{Prop}
If $(S^2,g)$ is a nontrivial soliton, then it can neither be a gradient soliton nor a soliton corresponding to a divergence-free vector field.
\end{Prop}
\begin{proof}
Taking the trace of the equation $(-R+r)g=L_Xg$, we get $(-R+r)=\mathrm{div}X$ and thus, the soliton must be trivial if $\mathrm{div}X=0$.

Now consider the case where $X$ is a gradient vector field.
For a conformal Killing vector field the zeros are isolated \cite{MR0350663}. Let $X$ be such that $(-R+r)g=L_Xg$ and let $p_1,\ldots p_k$ be the zeros of $X$. Let $v$ be a smooth function such that $\Delta v=\mathrm{div}V^{\flat}$ and $a_i:=v(p_i)\in\R$. Let $p\in S^2$ be arbitrary and $\varphi_t(p)$, $t\in \R$ be the integral curve of $X$ with $\varphi_0(p)=p$. Since $X$ is a gradient vector field, $\varphi_t(p)$ joins two of the $p_i$ and because of Lemma \ref{vconstant}, $v(p)=a_i$ for some $i\in \left\{1,\ldots,k\right\}$. Since $p$ was arbitrary, this implies that $v$ is piecewise constant and since $S^2$ is connected, $v$ must be constant. 

It follows that $R_g=R$ because $\mathrm{div}V=\Delta v=0$ so that $g$ is a soliton of the standard Ricci flow. But any such soliton must be trivial, see e.g. \cite[Proposition 5.21]{MR2061425}.
\end{proof}
\begin{Bem}
By continuity one can show that any vector field close to a gradient field cannot appear on the right hand side of the soliton equation \eqref{self-similar-solution}.
This again follows from Lemma \ref{vconstant} since the qualitative behavior of the dynamical system generated by \(X\) does not change
under small perturbations.
We conjecture that there exist no nontrivial solitons at all.
\end{Bem}

\bibliographystyle{plain}
\bibliography{mybib}

\begin{thebibliography}{10}

\bibitem{MR3457391}
Ilka Agricola and Margarita Kraus.
\newblock Manifolds with vectorial torsion.
\newblock {\em Differential Geom. Appl.}, 45:130--147, 2016.

\bibitem{MR2064965}
Paul Baird, Ali Fardoun, and Rachid Regbaoui.
\newblock The evolution of the scalar curvature of a surface to a prescribed
  function.
\newblock {\em Ann. Sc. Norm. Super. Pisa Cl. Sci. (5)}, 3(1):17--38, 2004.

\bibitem{MR2371700}
Arthur~L. Besse.
\newblock {\em Einstein manifolds}.
\newblock Classics in Mathematics. Springer-Verlag, Berlin, 2008.
\newblock Reprint of the 1987 edition.

\bibitem{MR0350663}
David~E. Blair.
\newblock On the zeros of a conformal vector field.
\newblock {\em Nagoya Math. J.}, 55:1--3, 1974.

\bibitem{MR3493217}
Volker Branding.
\newblock Dirac-harmonic maps with torsion.
\newblock {\em Commun. Contemp. Math.}, 18(4):1550064, 19, 2016.

\bibitem{MR908146}
Sun-Yung~Alice Chang and Paul~C. Yang.
\newblock Prescribing {G}aussian curvature on {$S^2$}.
\newblock {\em Acta Math.}, 159(3-4):215--259, 1987.

\bibitem{MR2061425}
Bennett Chow and Dan Knopf.
\newblock {\em The {R}icci flow: An introduction}, volume 110 of {\em
  Mathematical Surveys and Monographs}.
\newblock American Mathematical Society, Providence, RI, 2004.

\bibitem{MR954419}
Richard~S. Hamilton.
\newblock The {R}icci flow on surfaces.
\newblock In {\em Mathematics and general relativity ({S}anta {C}ruz, {CA},
  1986)}, volume~71 of {\em Contemp. Math.}, pages 237--262. Amer. Math. Soc.,
  Providence, RI, 1988.

\bibitem{hersch1970quatre}
Joseph Hersch.
\newblock Quatre propri{\'e}t{\'e}s isop{\'e}rim{\'e}triques de membranes
  sph{\'e}riques homogenes.
\newblock {\em CR Acad. Sci. Paris S{\'e}r. AB}, 270:A1645--A1648, 1970.

\bibitem{MR925006}
J{\"u}rgen Jost.
\newblock {\em Nonlinear methods in {R}iemannian and {K}\"ahlerian geometry},
  volume~10 of {\em DMV Seminar}.
\newblock Birkh\"auser Verlag, Basel, 1988.

\bibitem{MR2426247}
Jeffrey Streets.
\newblock Regularity and expanding entropy for connection {R}icci flow.
\newblock {\em J. Geom. Phys.}, 58(7):900--912, 2008.

\bibitem{MR2137948}
Michael Struwe.
\newblock A flow approach to {N}irenberg's problem.
\newblock {\em Duke Math. J.}, 128(1):19--64, 2005.

\bibitem{MR712664}
F.~Tricerri and L.~Vanhecke.
\newblock {\em Homogeneous structures on {R}iemannian manifolds}, volume~83 of
  {\em London Mathematical Society Lecture Note Series}.
\newblock Cambridge University Press, Cambridge, 1983.

\end{thebibliography}
\end{document}